\documentclass[11pt,a4paper]{scrartcl}

\usepackage[centertags]{amsmath} 
\usepackage{color} 
\usepackage{amsmath}
\usepackage{amsfonts}
\usepackage{amssymb}
\usepackage{amsthm}
\usepackage{epsfig} 
\usepackage{rotating}
\usepackage{multirow}
\usepackage{multirow}
\usepackage[table]{xcolor}
\usepackage{psfrag}
\usepackage{caption}
\usepackage{booktabs}
\usepackage{enumitem}

\newcommand{\R}{{\mathbb R}} 
\newcommand{\Nat}{\mathbb{N}} 

\newcommand{\X}{\mathcal{X}} 
\newcommand{\Z}{\mathcal{Z}} 
\newcommand{\U}{\mathcal{U}} 
\newcommand{\C}{\mathcal{C}} 

\newcommand{\List}{\mathcal{L}} 

\newcommand{\lb}[1]{\underline{#1}} 
\newcommand{\ub}[1]{\overline{#1}}

\newtheorem{thm}{Theorem}
\newtheorem{lem}[thm]{Lemma}
\newtheorem{prop}[thm]{Proposition}

\newtheorem{defn}{Definition}
\newtheorem{alg}{Algorithm}
\newtheorem{exmp}{Example}

\begin{document}

\begin{center}
% Title
{\huge
\bf
\textsf{Rigorous constraint satisfaction\\[2mm] for sampled linear systems}
}

\renewcommand{\thefootnote}{$\dagger$} 

\vspace{5mm}
{
Moritz Schulze Darup\footnotemark[1]
}
\vspace{2mm}

  \footnotetext[1]{M. Schulze Darup is with 
the Control Group, Department of Engineering Science,
        University of Oxford, Parks Road, Oxford OX1 3PJ, UK.
        E-mail: {\tt moritz.schulzedarup@rub.de}.}
\end{center}

\paragraph{Abstract.}
We address a specific but recurring problem related to sampled linear systems. In particular, we provide a numerical method for the rigorous verification of constraint satisfaction for linear continuous-time systems between sampling instances. The proposed algorithm combines elements of classical branch and bound schemes from global optimization with a recently published procedure to bound the exponential of interval matrices.


\section{Introduction and Problem Statement}
\label{sec:problemStatement}

We consider the continuous-time linear system
\vspace{-0.5mm}
\begin{equation}
\label{eq:sysCont}
\dot{x}(t) = A \, x(t) + B \, u(t), \qquad x(0)=x_0
\vspace{-0.5mm}
\end{equation}
with state and input constraints of the form 
\vspace{-0.5mm}
\begin{equation}
\label{eq:constraintsCont}
x(t) \in \X  \quad \text{and} \quad u(t) \in \U
  \quad \text{for every} \,\, t \in \R_0
  \vspace{-0.5mm}
 \end{equation}
under piecewise constant control
\vspace{-0.5mm}
\begin{equation}
\label{eq:uZOH}
u(t) = u(t_k)  \quad \text{for every} \,\, t \in [k\,\Delta t, (k+1)\,\Delta t),
\vspace{-0.5mm}
\end{equation}
where $\Delta t >0$ denotes the sampling time and where $t_k := k\,\Delta t$ for every $k \in \Nat$. The sets $\X\subset \R^n$ and $\U\subset \R^m$ are assumed to be convex and compact polytopes containing the origin as an interior point.
During controller design (and controller evaluation), system~\eqref{eq:sysCont} is usually replaced by the discrete-time system
\vspace{-0.5mm}
\begin{equation}
\label{eq:sysDisc}
x(t_{k+1}) = \widehat{A}\,x(t_k)+\widehat{B}\,u(t_k), \qquad x(0)=x_0
\vspace{-0.5mm}
\end{equation}
with $\widehat{A} := \exp(A\,\Delta t)$ and $\widehat{B}:=\int_{0}^{\Delta t} \exp(A\,\tau)  \, \mathrm{d}\tau \,B$. 
While the discretized system and the continuous-time system coincide at all sampling instances, it is well-known that the continuous-time trajectory may violate the state constraints even though the discrete-time counterpart does not (see, e.g., the motivating example
in~\cite{Sopasakis2014}).
This potential problem can be prevented by computing adapted constraints for the discretized system such that constraint satisfaction of~\eqref{eq:sysDisc} w.r.t.~the adapted constraints implies constraints satisfaction of~\eqref{eq:sysCont} w.r.t.~the original constraints~\eqref{eq:constraintsCont}  (see, e.g.,\cite{Sopasakis2014,Beradi2001_ECC,SchulzeDarup2015_CDC1}).

Comparing the methods  for the computation of adapted constraints in \cite{Sopasakis2014,Beradi2001_ECC,SchulzeDarup2015_CDC1}, it is peculiar that the procedures in \cite{Beradi2001_ECC} and \cite{SchulzeDarup2015_CDC1} both rely on similar non-convex optimization problems (cf. \cite[Thm.~5]{Beradi2001_ECC} and \cite[Eq. (15)]{SchulzeDarup2015_CDC1}). 
Roughly speaking, the underlying problem reads as follows. For a given state $x_0 \in \X$ and input $u_0 \in \U$ such that $\widehat{A} x_0 + \widehat{B} u_0 \in \X$ (i.e., the discretized systems satisfies the constraints), we are interested in checking whether the trajectory of the continuous-time systems violates the state constrains for some $t \in (0,\Delta t)$.

More formally, the problem of interest can be described along the following lines.
First note that the polytope $\X$ can be written in the form
$$
\X = \{ x \in \R^n \,|\, H x \leq \mathbf{1} \},
$$
where $H \in \R^{p \times n}$ and where $\mathbf{1} \in \R^p$ is a vector with all entries equal
to 1. Now, let $\varphi(t,x_0,u_0)$ denote the solution of~\eqref{eq:sysCont} at time $t \in [0,\Delta t]$ for an initial condition $x_0 \in \X$ and  a control action $u_0 \in \U$.
Then, the trajectory of the continuous-time system does obviously not violate the state constraints for any $t \in [0,\Delta t]$ if
\begin{equation}
\label{eq:conditionMaxSets}
  \max_{j\in \Nat_{[1,p]}} \max_{t\in[0,\Delta t]} e_j^T  H \, \varphi(t,x_0,u_0) \leq 1,
\end{equation}
where $e_j \in \R^p$ is the $j$-th Euclidean unit vector.
Taking into account that $\varphi(t,x_0,u_0) $ reads
\begin{equation}
\label{eq:phiTX0U0}
\varphi(t,x_0,u_0) = \exp(A\,t)\,x_0+ \int_{0}^{t} \exp(A\,\tau)  \, \mathrm{d}\tau \,B \, u_0
% &=  \left(\int_{0}^{t} \exp(A\,\tau)  \, \mathrm{d}\tau \right) (A\,x_0+B\,u_0) + x_0
\end{equation}
for every $t \in [0, \Delta t]$, it is easy to see that  $e_j^T H \varphi(t,x_0,u_0)$ is in general not concave (nor convex) in $t$. 
%(for fixed parameters $A$, $B$, $x_0$, $u_0$, and $e_j$). 
Hence, verifying whether~\eqref{eq:conditionMaxSets} holds (or not) is a multivariate non-convex  optimization problem (OP). 
Fortunately, the l.h.s.~in~\eqref{eq:conditionMaxSets} can be easily decomposed into $p$ univariate OPs of the form
\begin{equation}
\label{eq:maxF}
f^\ast := \max_{t \in [0,\Delta t]} f(t),
\end{equation}
where $f: [0,\Delta t] \rightarrow \R$ is given by
\begin{equation}
\label{eq:f}
f(t):= h^T \left(\exp(A\,t)\,x_0+ \int_{0}^{t} \exp(A\,\tau)  \, \mathrm{d}\tau \,B \, u_0\right)
\end{equation}
with $h\in \R^n$.
Clearly,~\eqref{eq:conditionMaxSets} holds if $f^\ast \leq 1$ results from~\eqref{eq:maxF} for every $h \in \{ H^T e_1, \dots, H^T e_P\}$.

Following the argumentation in~\cite{Beradi2001_ECC} (and \cite{SchulzeDarup2015_CDC1}), although~\eqref{eq:maxF} is non-convex, it can be solved reliable since it is the search of the maximum of a scalar function on a scalar compact domain. While this observation is true, we can provide more elaborated solution strategies for~\eqref{eq:maxF} based on the special structure of the objective function in~\eqref{eq:f}. 
In this paper,  we address the rigorous (or global) solution of~\eqref{eq:maxF} using interval arithmetic (IA). 
More precisely, we intend to identify non-decreasing, non-increasing,  convex and concave segments of $f(t)$ on $[0,\Delta t]$ based on interval inclusions for the first and second time-derivative of $f(t)$. Clearly, for such segments, local maxima can be easily computed and subsequently finding the global maximum is straightforward.

The paper is organized as follows.
We state basic notation and preliminaries in Sect.~\ref{sec:notation}.
The main result of the paper, i.e., a tailored branch and bound algorithm for the rigorous solution of~\eqref{eq:maxF}  is presented in Sect.~\ref{sec:computation}. 
Finally, the proposed method is illustrated with some  examples in Sect.~\ref{sec:example} before giving conclusions in Sect.~\ref{sec:Conclusion}.

\section{Notation and Preliminaries}
\label{sec:notation}

As mentioned in the introduction, we exploit IA to provide interval inclusions for $f(t)$ and its derivatives  
$$\frac{ \mathrm{d} f(t)}{\mathrm{d} t}  :=f^\prime(t) \qquad  \text{and} \qquad 
  \frac{ \mathrm{d}^2 f(t)}{\mathrm{d} t^2}  :=f^{\prime\prime}(t). $$
IA can be understood as the extension of operations associated with real numbers, like addition or multiplication, to intervals (see, e.g., \cite[Sect. 2.2]{Moore1979}).
In this paper, we only require a few interval operations summarized in the following lemma.

\begin{lem}[{\cite[Eqs. (2.14) and (2.19)]{Moore1979}}]
\label{lem:IA}
Let~$[c]\!=\![\lb{c},\ub{c}]$ $\subset \R$ and $[d]=[\lb{d},\ub{d}]\subset \R$ be intervals with $\lb{c}\leq \ub{c}$ and $\lb{d}\leq \ub{d}$. Define the intervals
\begin{align}
\nonumber
[c]+[d] &:= [\lb{c}+\lb{d},\ub{c}+\ub{d}] \quad \text{and} \\
\nonumber
[c]\times [d] &:= [\min\{\lb{c}\, \lb{d},\lb{c} \,\ub{d},\ub{c}\, \lb{d},\ub{c}\, \ub{d}\},\max\{\lb{c}\, \lb{d},\lb{c} \,\ub{d},\ub{c}\, \lb{d},\ub{c}\, \ub{d}\}].
\end{align}
Then, $c+d \in [c]+[d]$ and $c \, d \in [c] \times [d]$ for every $c \in [c]$ and every $d \in [d]$.
\end{lem}

The rules in Lem.~\ref{lem:IA} can also be applied to compute the sum (or the multiplication) of an interval $[c]$ and a real number $d\in \R$. In this case, $[d]$ can be construed as a degenerated interval with $\lb{d}=\ub{d}=d$.
Moreover, by setting $[d]=[c]$, the interval multiplication can be used to evaluate $[c]$ raised to the power of $\kappa \in \Nat$. However, tighter inclusions result for the calculation rule given in \cite[Eq. (3.10)]{Moore1979}. In fact, we find $c^\kappa \in [c]^\kappa$ for every $c \in  [c]$, where
$$
[c]^\kappa := \left\{ \begin{array}{ll@{\,\,}l}
\,[\lb{c}^\kappa, \ub{c}^\kappa] & \text{if} \,\,\, \lb{c}>0 & \text{or} \,\,\, \kappa \,\, \text{is odd},\\
\,[\ub{c}^\kappa,\lb{c}^\kappa] &\text{if} \,\,\, \ub{c}<0 & \text{and} \,\,\, \kappa \,\, \text{is even}, \\
 \,[0, |[c]|^\kappa] & \text{if} \,\,\, 0\in[c] & \text{and} \,\,\, \kappa \,\, \text{is even},
  \end{array} \right.  
  $$
  and where the magnitude of $[c]$ is defined as 
$|[c]|:=\max \{ |\lb{c}|, |\ub{c}|\}$. In addition, we define the width of an interval as $w([c]):=\ub{c}-\lb{c}$.
IA can be easily extended to interval vectors and interval matrices. 
For two interval matrices $[C]=[\lb{C},\ub{C}]$ and $[D]=[\lb{D},\ub{D}]$ of appropriate size, the sum $[C]+[D]$ and the multiplication $[C]\,[D]$ are understood component-wise.
Analogously, the magnitude $|[C]|$ is defined component-wise, i.e.,
$(|[C]|)_{ij}:=|[\lb{C}_{ij},\ub{C}_{ij}]|$.
Finally, the infinity norm of an interval matrix is defined as
 the maximum of the
norms of the contained real matrices, i.e.,
$\| [C] \|_\infty := \max_{C \in [C]} \| C \|_\infty$.
It is easy to see, that this definition implies  $\| [C] \|_\infty =  \left\| |[C]| \right\|_\infty$.
Computing interval inclusions for~\eqref{eq:f} will mainly build on interval inclusions for matrix exponentials, which can be calculated as follows.

\begin{thm}[{\cite[Thm. 4.3]{Goldsztejn2014}}]
\label{thm:intervalExp}
Let $[C]=[\lb{C},\ub{C}]$ be an interval matrix with $\lb{C},\ub{C} \in \R^{q \times q}$. Let $k,l \in \Nat$ be such that $2^l \,(k +2) > \| [C]\|_\infty$. Define  $[C^\ast]:= \frac{1}{2^l}[C]$,
\begin{align}
\nonumber
[D^\ast] &:=  I_q + \frac{[C^\ast]}{1}\,\left( I_q + \frac{[C^\ast]}{2} \left(  \dots \left( I_q + \frac{[C^\ast]}{k} \right) \dots\right)\right) \\
\nonumber
%\label{eq:intMatD}
& \quad + \frac{\| [C^\ast] \|_\infty^{k+1}}{(k+1)! \,(1-\frac{\| [C^\ast] \|_\infty}{k+2})} \, [-I_q,I_q],
\end{align}
and $[D]:=[D^\ast]^{2^l}$.
Then  $\exp(C) \in [D]$ for every $C \!\in [C]$.
\end{thm}

Note that there exist many ways to evaluate $[D^\ast]^{2^l}$ as occurring in Thm.~\ref{thm:intervalExp}. In \cite[p. 61]{Goldsztejn2014}, the authors propose to use $l$ successive 
interval square operations, i.e., 
$$[D^\ast]^{2^l} = \left( \dots \left( [D^\ast]^2 \right)^2 \dots\right)^2\!.$$
An efficient procedure for the computation of the square of an interval matrix is presented in~\cite[Sect. 6]{Kosheleva2005}.

\section{Rigorous Solution via Interval Arithmetic}
\label{sec:computation}

In the following, we present a tailored method for the rigorous solution of the non-convex OP~\eqref{eq:maxF}.
%The new method adapts and extends established branch-and-bound procedures for global optimization. 
Before describing the algorithm, we have to stress that there exists a number of situations
where~\eqref{eq:maxF} can be solved analytically. In this context,  note that~\eqref{eq:f} can be rewritten as 
\begin{equation}
\label{eq:fForANotInvertible}
f(t) = h^T \left( \int_{0}^{t} \!\exp(A\,\tau)  \, \mathrm{d}\tau \,(A\,x_0+B\,u_0)+x_0 \right)
\end{equation}
using the identity
$\int_{0}^{t} \exp(A\,\tau)  \, \mathrm{d}\tau \,A + I_n= \exp(A\,t)$. 
Obviously, trivial solutions result if $A\,x_0+B\,u_0=0$, $A=0$, or $h=0$.
In addition, an analytical solution of~\eqref{eq:maxF} is straightforward if 
$h$ is an eigenvector of $A^T$, i.e., if $h^T A = \lambda \, h^T$ for some $\lambda \in \R$.
To see this, note that the time-derivatives of~\eqref{eq:fForANotInvertible} are given by
\begin{align}
\label{eq:f1ForGeneralA}
f^\prime(t) &= h^T \exp(A\,t) \,(A\,x_0+B\,u_0) \quad \text{and} \\
\label{eq:f2ForGeneralA}
f^{\prime\prime}(t) &= h^T A \exp(A\,t) \,(A\,x_0+B\,u_0).
\end{align} 
Thus, $h$ being an eigenvector implies
 $f^{\prime\prime}(t)=\lambda f^\prime(t)$ for every $t \in [0,\Delta t]$, which eventually leads to a monotone function $f$. Consequently,  we obtain $f^\ast = \max \{h^T\! x_0, f(\Delta t)\}$.
Finally, the solution of~\eqref{eq:maxF} may be trivial if $A$ is nilpotent, i.e., if there exists an $r \in \Nat$ (with $1 \leq r \leq n$) such that $A^k \neq 0$ for $k \in \{0,\dots,r-1\}$ and $A^k = 0$ for~$k\geq r$. In this case,  $f$ can be rewritten as the polynomial
\begin{equation}
\label{eq:fNilpotent}
\!f(t)=  h^T \left( x_0  \!+\!\!\sum_{k=1}^{r-1}\! \frac{A^k x_0 + A^{k-1} B u_0}{k!} t^k + \!\frac{A^{r-1} B u_0}{r!} t^r\right). \!\!\!\! 
\end{equation}

If an analytical solution is not obvious, a numerical procedure to solve~\eqref{eq:maxF} may be required. 
We propose Alg.~\ref{alg:fMaxBranchAndBound} further below to compute $\epsilon$-optimal solutions to~\eqref{eq:maxF} according to Def.~\ref{defn:epsOpt}.

\begin{defn}
\label{defn:epsOpt}
Let $\epsilon\geq 0$. We call $\ub{f}^\ast \in \R$ an  \textit{$\epsilon$-optimal solution} to~\eqref{eq:maxF} if  $0 \leq \ub{f}^\ast - f^\ast \leq \epsilon$.
\end{defn}

As mentioned in the introduction, Alg.~\ref{alg:fMaxBranchAndBound} relies on identifying non-decreasing, non-increasing,  convex and concave segments of $f(t)$ on $[0,\Delta t]$ based on interval inclusions for the derivatives~\eqref{eq:f1ForGeneralA} and~\eqref{eq:f2ForGeneralA}.
As stated in the following proposition, such inclusions can be easily computed based on Thm.~\ref{thm:intervalExp}.

\begin{prop}
\label{prop:IntervalInclusions}
Let $[t] \subseteq [0, \Delta t]$ with $|[t]|>0$ and consider the interval matrix $[C]=A \,[t]$.  Let $k,l \in \Nat$ be such that $2^l \,(k +2) > \| [C]\|_\infty$ and define $[D]$ as in Thm.~\ref{thm:intervalExp}. Then, the interval inclusions
\begin{align}
\label{eq:f1Interval}
  [f^\prime] &=[\lb{f}^\prime,\ub{f}^\prime] & \!\!\!\!\! = \,\,& h^T [D] \, (A x_0 + B \,u_0), \quad \text{and}\\
\label{eq:f2Interval}
   [f^{\prime\prime}] & =[\lb{f}^{\prime\prime},\ub{f}^{\prime\prime}] & \!\!\!\!\! = \,\, & h^T A \, [D] \, (A x_0 + B \,u_0),
\end{align}
are such that
\begin{equation}
\label{eq:inclusionF1andF2}
f^\prime(t) \in [f^\prime]\,\,\,   \text{and} \,\,\,
f^{\prime\prime}(t) \in [f^{\prime\prime}] \quad \text{for every} \,\,\,t\in [t].\!
\end{equation}
\end{prop}

\begin{proof}
According to Thm.~\ref{thm:intervalExp}, we have $\exp(A t) \in [D]$ for every $t \in [t]$. Thus,  \eqref{eq:f1Interval} and~\eqref{eq:f2Interval} contain the r.h.s.~in~\eqref{eq:f1ForGeneralA} and \eqref{eq:f2ForGeneralA} for  every $t \in [t]$, respectively. Consequently, \eqref{eq:inclusionF1andF2} holds.  
\end{proof}

Clearly, if $\lb{f}^\prime \geq 0$ results from~\eqref{eq:f2Interval}, then $f(t)$ is non-decreasing on the time-interval $[t]$. Analogously, $\ub{f}^\prime \leq 0$, $\lb{f}^{\prime\prime} \geq 0$, or $\ub{f}^{\prime\prime} \leq 0$ guarantees $f(t)$ to be non-increasing, convex, or concave on $[t]$, respectively. In each of these cases, it is easy to compute the local maximum of $f(t)$ on $[t]$, i.e.,
\begin{equation}
\label{eq:fMaxLocal}
f^\dagger := \max_{t\in[t]} f(t)
\end{equation}
In fact, $f(t)$ being convex, non-decreasing, or non-increasing implies $f^\dagger = \max \{ f(\lb{t}), f(\ub{t})\}$, $f^\dagger = f(\lb{t})$, or $f^\dagger = f(\ub{t})$. Finally, if $f(t)$ is concave, solving~\eqref{eq:fMaxLocal} is a convex OP. 
In contrast, if $\lb{f}^\prime<0<\ub{f}^\prime$ and $\lb{f}^{\prime\prime} < 0<\ub{f}^{\prime\prime}$,  
 a straightforward computation of $f^\dagger$ may not be possible. However, even in this case, the bounds on the derivatives can be used to compute an upper bound for the local maximum according to Def.~\ref{defn:overestimator} and Lem.~\ref{lem:Overestimators}.

\begin{defn}
\label{defn:overestimator}
Let $[t] \subseteq [0,\Delta t]$ with $w([t])>0$. We call a function $g:[t] \rightarrow \R$ a \textit{suitable overestimator} for $f$ on $[t]$ if $f(t) \leq g(t)$ for every $t \in [t]$ and if the optimizer 
\begin{equation}
\label{eq:tAstOverestimator}
t^\dagger:= \arg \max_{t \in[t]} g(t)
\end{equation}
 can either be computed analytically or by solving a convex optimization problem.
\end{defn}

\begin{lem}
\label{lem:Overestimators}
Let $[t] \subseteq [0,\Delta t]$ with $w([t])>0$ and assume $[f^{\prime}]$ and $[f^{\prime\prime}]$ with $\lb{f}^\prime<0<\ub{f}^\prime$ and $\ub{f}^{\prime\prime}>0$ are such that~\eqref{eq:inclusionF1andF2} holds.
Then, the following three functions $g:[t] \rightarrow \R$ are suitable overestimations for $f$ on $[t]$.

\begin{enumerate}[itemsep=0mm,leftmargin=4mm,itemindent=0mm,labelsep=1mm]
\item The piecewise affine function
$$
g(t) := \left\{ \begin{array}{ll}
f(\lb{t}) + \ub{f}^\prime (t - \lb{t}) &\text{if} \,\,\,  t \leq t_c, \\
f(\ub{t}) - \lb{f}^\prime (\ub{t} -t ) & \text{otherwise},
\end{array} \right. 
$$ 
where  $t_c= \frac{\ub{f}^\prime \lb{t} - \lb{f}^\prime \ub{t}+f(\ub{t})-f(\lb{t})}{\ub{f}^\prime - \lb{f}^\prime}.$

\item The piecewise quadratic function
$$
g(t) := \left\{ \begin{array}{ll}
f(\lb{t}) + f^\prime(\lb{t})\,(t - \lb{t})+\frac{\ub{f}^{\prime\prime}}{2} (t - \lb{t})^2  & \text{if} \,\,\, t \leq t_c,\\
f(\ub{t}) - f^\prime(\ub{t})\,(\ub{t}-t)+ \frac{\ub{f}^{\prime\prime}}{2} (\ub{t}-t)^2 & \text{otherwise},
\end{array} \right. 
$$
where 
$$\!\!t_c:= \left\{ \!\begin{array}{ll} \frac{0.5\,\ub{f}^{\prime\prime} (\ub{t}^2-\lb{t}^2) +f^\prime(\lb{t})\, \lb{t}-f^\prime(\ub{t})\,\ub{t} +f(\ub{t})-f(\lb{t})}{\ub{f}^{\prime\prime} (\ub{t}-\lb{t}) + f^\prime(\lb{t})  - f^\prime(\ub{t})} & \text{if} \,\,\,\frac{f^\prime(\ub{t})- f^\prime(\lb{t})}{\ub{t}-\lb{t}} \!< \ub{f}^{\prime\prime}\!\!\!,   \\
\ub{t} & \text{otherwise}. \end{array} \right.$$

\item The concave function
 $
g(t) := f(t) + \frac{\ub{f}^{\prime\prime}}{2} (t - \lb{t})\, (\ub{t}-t).
$
\end{enumerate}
\end{lem}

The overestimators listed in Lem.~\ref{lem:Overestimators} are adopted from \cite{Piyavskii1972}, \cite{Breiman1993}, and
\cite[Sect. 4]{Maranas1994b}. In fact, $|[f^\prime]|$ and $|[f^{\prime\prime}]|$ can be understood as local Lipschitz constants for $f(t)$ and $f^\prime(t)$ as exploited in \cite{Piyavskii1972} and \cite{Breiman1993}, respectively. We thus omit a detailed proof of Lem.~\ref{lem:Overestimators} and refer to \cite{Piyavskii1972,Breiman1993,Maranas1994b}. It is, however, important to note that the solution to~\eqref{eq:tAstOverestimator} reads $t^\dagger = t_c$ for the overestimator $g$ of type 1. 
For type~2, we find $t^\dagger \in \{\lb{t},t_c,\ub{t}\}$, which renders~\eqref{eq:tAstOverestimator} trivial. Finally, for type 3, solving~\eqref{eq:tAstOverestimator} is a convex OP. 
Based on Prop.~\ref{prop:IntervalInclusions} and Lem.~\ref{lem:Overestimators}, we are finally able to formulate an algorithm for the computation of
an $\epsilon$-optimal solution to~\eqref{eq:maxF}.

\begin{alg}
\label{alg:fMaxBranchAndBound}
Solution of~\eqref{eq:maxF} via branch and bound.

\begin{enumerate}[itemsep=0mm,leftmargin=5mm,itemindent=0mm,labelsep=1.5mm,
parsep=0.5mm,topsep=0mm]

\item 
Initialize the lower bound on the global maximum as $\lb{f}^\ast \leftarrow h^T x_0$.
Initialize the list $\List$ of tuples $([t],[f^\dagger])$, each containing a time-interval $[t]$ and bounds $[f^\dagger]$ on the local maximum of $f$ on $[t]$, as $\List \leftarrow \{([0,\Delta t], [-\infty,\infty])\}$.

\item \textbf{for each}  tuple $([t],[f^\dagger])$ in $\List$, for which the bounds on the local maximum read $[f^\dagger]=[-\infty,\infty]$, repeat the following steps.
\begin{enumerate}[itemsep=0mm,leftmargin=4mm,itemindent=0mm,labelsep=1mm,
parsep=0.5mm,topsep=0mm]

\item Compute $[f^\prime]$ and $[f^{\prime\prime}]$ according to Prop.~\ref{prop:IntervalInclusions} and define a suitable overestimator $g$ for $f$ on $[t]$ (e.g., according to Lem.~\ref{lem:Overestimators}).

\item  \textbf{if} $\lb{f}^\prime \geq 0$, set $[f^\dagger] \leftarrow [f(\ub{t}),f(\ub{t})]$.

  \textbf{else if} $\ub{f}^\prime \leq 0$, set $[f^\dagger] \leftarrow [f(\lb{t}),f(\lb{t})]$.

\textbf{else if} $\lb{f}^{\prime\prime}\! \geq  0$, compute $f^\dagger = \max \{f(\lb{t}),f(\ub{t})\}$ and set
$[f^\dagger] \leftarrow [f^\dagger ,f^\dagger ]$.

\textbf{else if} $\ub{f}^{\prime\prime}\! \leq 0$, solve~\eqref{eq:fMaxLocal} and set $[f^\dagger] \leftarrow [f^\dagger ,f^\dagger ]$.

\textbf{else}, solve~\eqref{eq:tAstOverestimator} and set
$[f^\dagger] \leftarrow [f(t^\dagger),g(t^\dagger)]$.

\item \textbf{if} $\lb{f}^\dagger > \lb{f}^\ast$, set $\lb{f}^\ast \leftarrow \lb{f}^\dagger$.

\end{enumerate}

\item Compute the upper bound $\ub{f}^\ast$ on the global maximum by taking the maximum of all local upper bounds $\ub{f}^\dagger$ of the tuples $([t],[f^\dagger])$ in $\List$.

\item \textbf{if} $w([f^\ast]) \leq \epsilon$,  \textbf{return} $\ub{f}^\ast$ and terminate.

\item \textbf{for each}  tuple  in $\List$ repeat the following step.
\begin{enumerate}[itemsep=0mm,leftmargin=5mm,itemindent=0mm,labelsep=1mm,
parsep=0.5mm,topsep=0mm]

\item \textbf{if} $\ub{f}^\dagger\! \leq \lb{f}^\ast$ and $w([f^\dagger]) \!> \epsilon$, remove tuple from $\List$. 

\end{enumerate}
\item Select the  tuple  $([t],[f^\dagger])$ with the largest width $w([f^\dagger])$ in $\List$ and remove it from $\List$. Compute $t_m = \frac{\lb{t}+\ub{t}}{2}$ and insert the tuples $([\lb{t},t_m],[-\infty,\infty])$ and $([t_m,\ub{t}],[-\infty,\infty])$ in $\List$. \textbf{go to} step 2.
\end{enumerate}

\end{alg}

In principle, Alg.~\ref{alg:fMaxBranchAndBound} is similar to established branch and bound procedures for global optimization (see, e.g., \cite{Piyavskii1972}, \cite[Sects. 6 to 13]{Hansen1979}, \cite[Sect. 3]{Breiman1993},  \cite[Sect. 6]{Maranas1994b}, or
\cite[Sect. 3]{Sergeyev1998}).
The main difference is that Alg.~\ref{alg:fMaxBranchAndBound} makes use of bounds on the first \textit{and} second derivative.
First, this allows to identify a number of segments where the local maximum can be computed exactly. Second, it gives some flexibility w.r.t.~the choice of suitable overestimators for the remaining segments. In fact, overestimators of type 1 (in Lem.~\ref{lem:Overestimators}) depend on $[f^\prime]$ while type 2 and 3 build on $[f^{\prime\prime}]$. 
Regarding the computational effort, the strategy to compute both interval inclusions may be inefficient in general.
Here, however, the simultaneous calculation of $[f^\prime]$ \textit{and} $[f^{\prime\prime}]$ does not significantly increase the computational load compared to solely calculating $[f^\prime]$ \textit{or} $[f^{\prime\prime}]$. In fact, due to the special structure of $f$, we easily evaluate $[f^\prime]=h^T [d]$ and $[f^{\prime\prime}]=h^T \!A\, [d]$ given the interval vector $[d]:=[D]\,(A x_0 + B u_0)$. Obviously, the computational effort to calculate $[d]$ is dominated by the computation of the interval inclusion $[D]$ for the matrix exponential.  

As stated in Prop.~\ref{prop:AlgEps}, Alg.~\ref{alg:fMaxBranchAndBound} is guaranteed to compute an $\epsilon$-optimal solution to~\eqref{eq:maxF} for every $\epsilon>0$.
In many cases, however, Alg.~\ref{alg:fMaxBranchAndBound} is capable to solve \eqref{eq:maxF} exactly, i.e., for $\epsilon=0$ (see Exmps.~\ref{exmp:1} through~\ref{exmp:3} in Sect.~\ref{sec:example}).

\begin{prop}
\label{prop:AlgEps}
Let $\epsilon>0$ and let $k, l \in \Nat$ be such that $2^l \,(k +2) > \| A \,[0,\Delta t] \|_\infty$. Then 
 Alg.~\ref{alg:fMaxBranchAndBound} terminates after finite time and returns an $\epsilon$-optimal solution to~\eqref{eq:maxF}.
\end{prop}

\begin{proof}
It is easy to see that Alg.~\ref{alg:fMaxBranchAndBound} provides an $\epsilon$-optimal whenever it terminates. Hence, it is sufficient to prove finite termination of the algorithm.
Clearly, Alg.~\ref{alg:fMaxBranchAndBound} terminates if (but not only if) we have $w([f^\dagger])\leq\epsilon$ for every tuple $([t],[f^\dagger])$ in the list $\List$. 
%In fact, every local inclusion  $[f^\dagger]$ can be written as $[f^\dagger]=[\lb{f}^\dagger, \lb{f}^\dagger + \Delta f^\dagger]$ with $\Delta f^\dagger \leq \epsilon$.
In fact, the upper bound on the global maximum then satisfies
$$
\ub{f}^\ast = \max_{([t],[f^\dagger]) \in \List} \ub{f}^\dagger
\leq \max_{([t],[f^\dagger]) \in \List} \lb{f}^\dagger + \epsilon = \lb{f}^\ast + \epsilon,$$
i.e., $w([f^\ast]) \leq \epsilon$.
As a direct consequence, the time-interval $[t]$ of a tuple $([t],[f^\dagger])$ satisfying $w([f^\dagger])\leq\epsilon$ will never be bisected in step 6 of Alg.~\ref{alg:fMaxBranchAndBound} (since this would contradict reaching step 6 after passing step 4 without termination). In the following, denote by $[f_0^\prime]$ and $[f_0^{\prime\prime}]$ the interval inclusions for $f^\prime$ and $f^{\prime\prime}$ on $[0,\Delta t]$ and let $j \in \Nat$ be such that
\begin{equation}
\label{eq:maxLeqEpsilon}
\max \left\{  w([f^\prime_0])\,\Delta \tau, \frac{w([f^{\prime\prime}_0])}{2} \, \Delta \tau^2\!, \frac{\ub{f}^{\prime\prime}_0}{8} \Delta \tau^2\right\} \leq \epsilon,
\end{equation}
where $\Delta \tau := \frac{\Delta t}{2^{j}}$. We obviously have
\begin{equation}
\label{eq:bisectionHeightJ}
[0,\Delta t] = \bigcup_{i=0}^{2^j-1} [i,i+1] \Delta \tau
\end{equation}
by construction. Consider any $i \in \{0,\dots,2^{j}-1\}$, set $[t]=[i,i+1]\,\Delta \tau$, and note that $w([t])=\Delta \tau$.
Further note that the  inclusions $[f^\prime]$ and $[f^{\prime\prime}]$ on $[t]$  satisfy $[f^\prime] \subseteq [f_0^\prime]$ and $[f^{\prime\prime}] \subseteq [f_0^{\prime\prime}]$ since  $[t]\subseteq [0,\Delta t]$ (and since all involved operations are \textit{inclusion increasing}; see \cite{Goldsztejn2014} for details). 
Now assume an overestimator of type 1 (as in Lem.~\ref{lem:Overestimators}) is applied.
We then find
\begin{align*}
g(t^\dagger) - f(t^\dagger) &\leq \max_{t \in [t]} g(t)-f(t) \\
& = \max_{t \in [t]} f(\lb{t})+ \ub{f}^\prime (t - \lb{t})-f(t) \\
& \leq \max_{t \in [t]} f(\lb{t})+ \ub{f}^\prime (t - \lb{t})-f(\lb{t})- \lb{f}^\prime (t - \lb{t}) \\
& = w([f^\prime]) \, w([t]) \leq w([f_0^\prime]) \, \Delta \tau \leq \epsilon,
\end{align*}
where the first and second relation hold due to $t^\dagger \in [t]$ and by definition of $g$, respectively. The third relation holds since 
$$
f(t) = f(\lb{t}) + \int_{\lb{t}}^t f^\prime(\tau) \,\mathrm{d} \tau \geq f(\lb{t})+\lb{f}^\prime (t - \lb{t})
$$
for every $t \in [t]$. Finally, the last relations hold due to $[f^\prime] \subseteq [f_0^\prime]$ and according to~\eqref{eq:maxLeqEpsilon}. Using analogous arguments, we obtain 
$g(t^\dagger) - f(t^\dagger) \leq \epsilon$
 also for overestimators of type 2 or 3. 
We thus find $w([f^\dagger])\leq \epsilon$ for the bounds on the local maximum of $f$ on $[t]$ according to step 2.(b) of Alg.~\ref{alg:fMaxBranchAndBound}. Since $i \in \{0,\dots,2^j-1\}$ was arbitrary, this observation holds for every time interval $[i,i+1]\Delta \tau$ on the r.h.s.~of~\eqref{eq:bisectionHeightJ}.
As a consequence, the number of required bisections in step 6 of Alg.~\ref{alg:fMaxBranchAndBound} is limited and the algorithm terminates after finite time. To see this, first note that $j$ and $i$ can be understood as the height and the position of a leaf node in a perfect binary tree, respectively. The binary tree can be associated with the bisection procedure. In fact, every inner node can be linked to the bisection of a time-interval. Now, the perfect binary tree with height $j$ refers to the worst-case scenario, where the bisection continues until we obtain the partition on the r.h.s.~of~\eqref{eq:bisectionHeightJ}. Since this tree contains $\sum_{i=0}^{j-1} 2^i = 2^{j} - 1$ inner nodes, we obtain a maximum of $2^j-1$ bisections.
\end{proof}

\section{Numerical Examples}
\label{sec:example}

We analyze four examples in the following.
The first two examples address technical systems taken from~\cite{Gutman1987} and~\cite{Sopasakis2014}. In contrast, Exmp.~\ref{exmp:3} and~\ref{exmp:4} are of academic nature. In fact, these examples were purely designed to challenge Alg.~\ref{alg:fMaxBranchAndBound}.

The application of Alg.~\ref{alg:fMaxBranchAndBound} requires to specify an error bound $\epsilon$. Moreover, the underlying computation of interval inclusions for matrix exponentials depends on the parameters $k,l \in \Nat$ (see Thm.~\ref{thm:intervalExp}). We set $\epsilon = 10^{-6}$ and $k=l=10$ for all examples.

\vspace{3mm}
\begin{figure}[htp] 

\psfrag{a}[c][c]{\footnotesize(a)}
\psfrag{b}[c][c]{\footnotesize(b)}
\psfrag{c}[c][c]{\footnotesize(c)}
\psfrag{d}[c][c]{\footnotesize(d)}
\psfrag{e}[c][c]{\footnotesize(e)}
\psfrag{F}[c][c]{\footnotesize(f)}
\psfrag{g}[c][c]{\footnotesize(g)}
\psfrag{h}[c][c]{\footnotesize(h)}

\psfrag{t}[c][c]{\footnotesize$t$}
\psfrag{f}[ct][cb]{\footnotesize$f(t)$}

\psfrag{x1}[c][c]{\footnotesize$x_1$}
\psfrag{x2}[ct][cb]{\footnotesize$x_2$}
\psfrag{x3}[rt][cb]{\footnotesize$x_3$}
\psfrag{xP}[c][c]{\footnotesize projection on $x_1$-$x_2$}

\psfrag{x11}[c][c]{\footnotesize$0.0$}
\psfrag{x12}[c][c]{\footnotesize$0.5$}
\psfrag{x13}[c][c]{\footnotesize$1.0$}

\psfrag{y11}[c][c]{\footnotesize$1.000$\,}
\psfrag{y12}[c][c]{\footnotesize$1.005$\,}

\psfrag{x21}[c][c]{\footnotesize$25.0$}
\psfrag{x22}[c][c]{\footnotesize$25.1$}

\psfrag{y21}[c][c]{\footnotesize$-0.5$}
\psfrag{y22}[cr][cr]{\footnotesize$0.0$}
\psfrag{y23}[cr][cr]{\footnotesize$0.5$}

\psfrag{x31}[c][c]{\footnotesize$0.00$}
\psfrag{x32}[c][c]{\footnotesize$0.25$}
\psfrag{x33}[c][c]{\footnotesize$0.50$}

\psfrag{y31}[c][c]{\footnotesize$0.94$}
\psfrag{y32}[c][c]{\footnotesize$1.00$}

\psfrag{x41}[c][c]{\footnotesize$-1.0$}
\psfrag{x42}[c][c]{\footnotesize$-0.5$}
\psfrag{x43}[c][c]{\footnotesize$ 0.0$}

\psfrag{y41}[c][c]{\footnotesize$-2.0$}
\psfrag{y42}[c][c]{\footnotesize$-1.8$}

\psfrag{x51}[c][c]{\footnotesize$0.0$}
\psfrag{x52}[c][c]{\footnotesize$0.5$}
\psfrag{x53}[c][c]{\footnotesize$1.0$}

\psfrag{y51}[cr][cr]{\footnotesize$-1.0$\!\!}
\psfrag{y52}[cr][cr]{\footnotesize$1.0$}

\psfrag{x61}[c][c]{\footnotesize$-1.0$}
\psfrag{x62}[c][c]{\footnotesize$0.0$}
\psfrag{x63}[c][c]{\footnotesize$1.0$}

\psfrag{y61}[c][c]{\footnotesize$-1.0$}
\psfrag{y62}[cr][cr]{\footnotesize\,\,$0.0$}
\psfrag{y63}[cr][cr]{\footnotesize$1.0$}

\psfrag{x71}[c][c]{\footnotesize$0.0$}
\psfrag{x72}[c][c]{\footnotesize$0.1$}
\psfrag{x73}[c][c]{\footnotesize$0.2$}

\psfrag{y71}[cr][cr]{\footnotesize$0.5$\!}
\psfrag{y72}[cr][cr]{\footnotesize$1.0$}

\psfrag{x81}[c][c]{\footnotesize$0.0$}
\psfrag{x82}[c][c]{\footnotesize$2.0$}
\psfrag{x83}[c][c]{\footnotesize$4.0$}

\psfrag{y81}[cr][cr]{\footnotesize$0.0$\!\!}
\psfrag{y82}[cr][cr]{\footnotesize$0.2$}

\begin{center}
\includegraphics[width=0.72\linewidth]{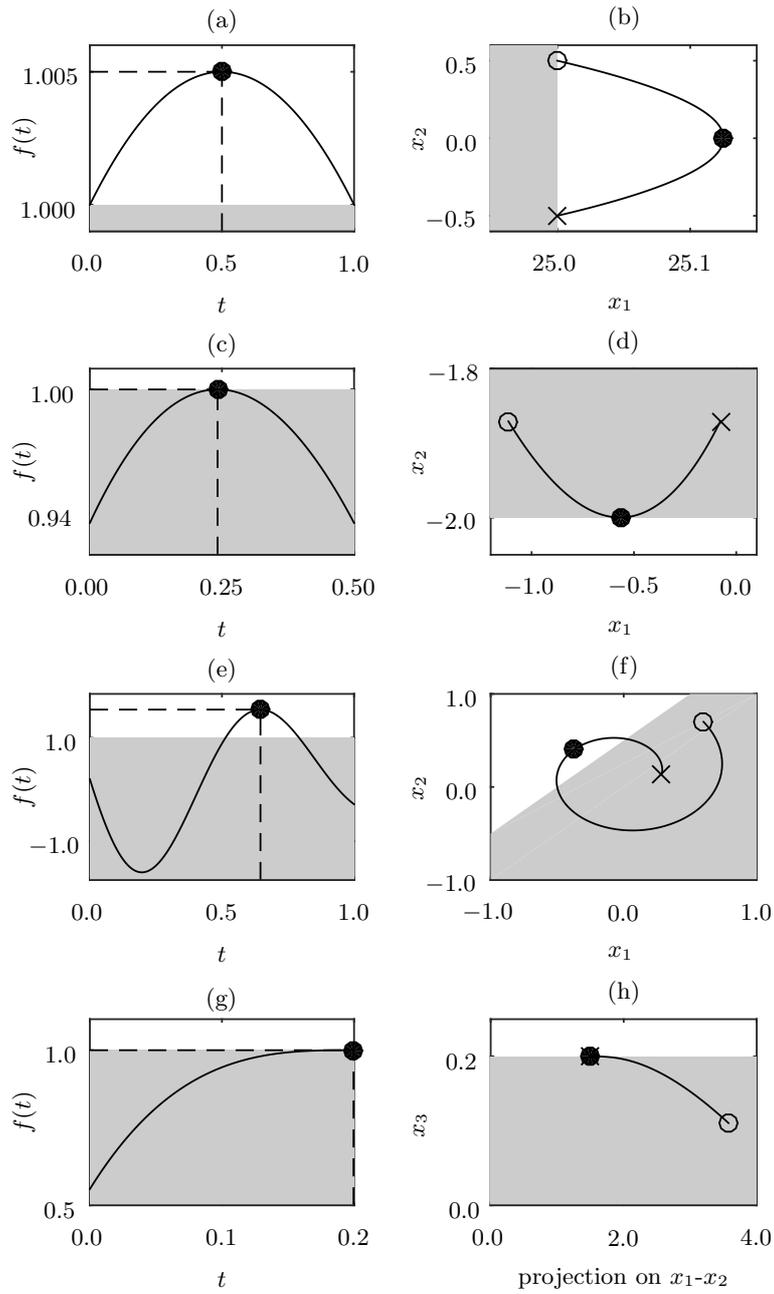}
\end{center}

\caption{
Illustration of $f(t)$ (left figures) and $\varphi(t)$ (right figures) for Exmps.~\ref{exmp:1} through \ref{exmp:4} (from top to bottom). In each figure, the point where the maximum $f^\ast$ is attained is marked with a filled circle. Open circles and crosses refer to initial states $x_0$ and final states $\varphi(\Delta t)$, respectively.
State constraints are violated outside the gray regions.
} 
\label{fig:example1to4}
\end{figure}

\begin{exmp}
\label{exmp:1}
We first analyze the double integrator in \cite{Gutman1987}
with the system matrices
$$A = \begin{pmatrix} 0 & 1 \\ 0 & 0  \end{pmatrix} \quad \text{and} \quad
B = \begin{pmatrix} 0 \\ 1  \end{pmatrix}$$
and the constraints $\X=\{x \in \R^2\,|\, |x_1| \leq 25, \,|x_2| \leq 5\}$ and $\U =[-1,1]$.  
As in \cite{Gutman1987}, we consider the sampling time $\Delta t =1$
and obtain the discretized system matrices
$$
\widehat{A} = \begin{pmatrix} 1 & 1 \\ 0 & 1  \end{pmatrix} \quad \text{and} \quad
\widehat{B} = \begin{pmatrix} 0.5 \\ 1.0  \end{pmatrix}.
$$
Obviously, for the initial state
$x_0 = (\, 25.0 \,\,\,\, 0.5 \,)^T \in \X$,  the only input $u_0 \in \U$ for which the discretized system satisfies the state constraints at the next sampling instant, i.e., for which $\widehat{A} x_0 + \widehat{B} u_0 \in \X$, is $u_0=-1$. In fact, for any $u_0 \in (-1,1]$, the state constraint $x_1 \leq 25$ will be violated. However, even for the choice $u_0=-1$,  
the continuous-time system may violate the state constraints for some $t \in (0,\Delta t)$.
To check whether the constraint $x_1 \leq 25$ will be violated (or not), we set
  $h = (\, 0.04 \,\,\,\, 0.00 \,)^T$ and solve~\eqref{eq:maxF}.
 Clearly, since  $A$ is nilpotent, \eqref{eq:maxF} can be easily solved analytically. We initially ignore this observation and apply Alg.~\ref{alg:fMaxBranchAndBound}.
 
Following the steps in Alg.~\ref{alg:fMaxBranchAndBound}, we first initialize the lower bound for the global maximum as $\lb{f}^\ast = h^T x_0 = 1$ and the list of tuples as $\List=\{([0,\Delta t],[-\infty,\infty])\}$. Since $[f^\dagger]=[-\infty,\infty]$, we then evaluate inclusions for $f^\prime$ and $f^{\prime\prime}$ on $[0,\Delta t]$ in step 2.(a) and obtain the (exact) intervals
  $$
  [f^\prime]=[-0.02,0.02]  \quad \text{and} \quad [f^{\prime\prime}]=[-0.04,-0.04].
  $$
  Since $\ub{f}^{\prime\prime}\leq 0$, the algorithm recognizes that $f$ is concave in step 2.(b), solves the convex OP 
$f^\dagger = \max_{ t \in [0,\Delta t]} f(t) = 1.005$,
  and sets $[f^\dagger] = [f^\dagger,f^\dagger]$.
Now, due to $\lb{f}^\dagger =1.005> \lb{f}^\ast$, the lower bound on the global maximum is updated in step 2.(c). 
Since $([0,\Delta t],[1.005,1.005])$ is the only tuple in $\List$, we move to step 3 and set $\ub{f}^\ast=1.005$. Finally, the algorithm terminates in step 4 since  $w([f^\ast]) = 0 \leq \epsilon$. 

For this example, it is easy to verify the computed result by analytically solving~\eqref{eq:maxF}. In fact, since $A$ is nilpotent  with degree $r=n=2$, we obtain 
  \begin{align}
\nonumber
f(t) &= h^T x_0 + h^T (A x_0 + B u_0)\,t + 0.5\,h^T \!A B u_0 \, t^2 \\
\nonumber
& =1 + 0.02 \, t -0.02 \,t^2
\end{align}
according to~\eqref{eq:fNilpotent}. We thus find $f^\ast=f(0.5)=1.005 = \ub{f}^\ast = \lb{f}^\ast$. Clearly, since $f^\ast>1$, the continuous-time system will violate the state constraint for some (here all) $t \in (0,\Delta)$. This can also be observed in Figs.~\ref{fig:example1to4}.(a) and \ref{fig:example1to4}.(b), where
 $f(t)$ and $\varphi(t)$ are illustrated, respectively. 
\end{exmp}

\begin{exmp}
\label{exmp:2}
We consider the example in \cite{Sopasakis2014} with
$$A = \begin{pmatrix} -0.7 & \textcolor{white}{+}0.1 \\ \textcolor{white}{+}2.0 & -0.1  \end{pmatrix} \quad \text{and} \quad
B = \begin{pmatrix}  2.0  \\ 1.0  \end{pmatrix}$$
plus $\X = \{x \in \R^n \,| \, \| x \|_\infty \leq 2 \}$ and $\U=[-1,1]$. As in~\cite{Sopasakis2014}, the sampling time is chosen as $\Delta t = 0.5$. We analyze whether the continuous-time system violates the constraint $x_2 \geq -2$ for the initial state
$x_0 = (\, -1.1135  \,\,\,\, -1.8708  \,)^T$ and the input
$u_0=0.9355$. To this end, we solve~\eqref{eq:maxF} with 
$h = (\,0.0  \,\,\,\, -0.5  \,)^T$ and obtain
$\ub{f}^\ast = \lb{f}^\ast = 0.9999$ using Alg.~\ref{alg:fMaxBranchAndBound}. Thus, the continuous-time system does not violate the state constraint $x_2 \geq -2$ for any $t \in [0,\Delta]$. This observation is important, since $(\,x_0 \,\,\, u_0\,)^T$ marks a vertex of the adapted constraint set $\widehat{\Z}$ as computed in~\cite[Sect. IV]{SchulzeDarup2015_CDC1}. In other words, $f^\ast \leq 1$ is required to confirm the results in~\cite{SchulzeDarup2015_CDC1}.
An illustration of $f(t)$ and $\varphi(t)$ can be found in Figs.~\ref{fig:example1to4}.(c) and \ref{fig:example1to4}.(d), respectively.
\end{exmp}

\begin{figure}[htp]

\psfrag{t}[c][c]{\footnotesize$t$}
\psfrag{f}[ct][cb]{\footnotesize$f(t)$}

\psfrag{x11}[c][c]{\footnotesize\,\,$0.0$}
\psfrag{x12}[c][c]{\footnotesize$0.5$}
\psfrag{x13}[c][c]{\footnotesize$1.0$}

\psfrag{y11}[c][c]{\footnotesize\,\,\,\,$-2$}
\psfrag{y12}[cr][cr]{\footnotesize$0$}
\psfrag{y13}[cr][cr]{\footnotesize$2$}
\psfrag{y14}[cr][cr]{\footnotesize$4$}
\psfrag{y15}[cr][cr]{\footnotesize$6$}

\begin{center}
\includegraphics[width=0.75\linewidth]{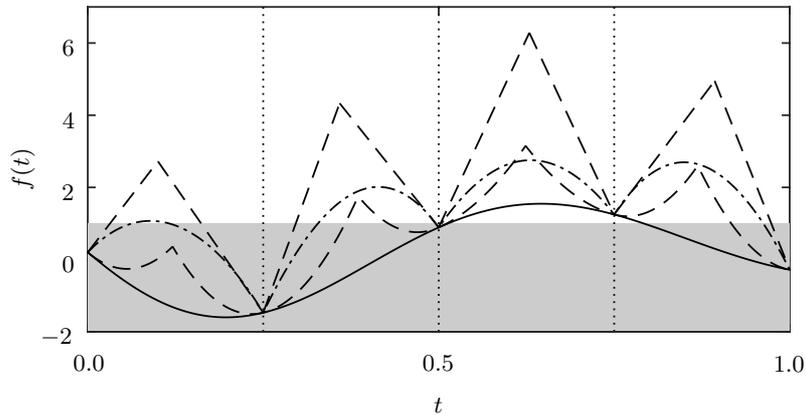}
\end{center}

\caption{
Illustration of some overestimators for $f(t)$ as in Exmp.~\ref{exmp:3} after three bisections in Alg.~\ref{alg:fMaxBranchAndBound}. The dashed lines refer to the piecewise linear and quadratic overestimators as introduced in Lem.~\ref{lem:Overestimators} (type 1 and 2), respectively. The dash-dotted curves show the concave overestimators (type 3).
} 
\label{fig:example3Details}
\end{figure}

\begin{exmp}
\label{exmp:3}
We consider the system matrices
$$A = \begin{pmatrix} -1 & \textcolor{white}{+}7 \\ -7 & -1  \end{pmatrix} \quad \text{and} \quad
B = \begin{pmatrix}  -1  \\ \textcolor{white}{+}0  \end{pmatrix},$$
the constraints $\X = \{x \in \R^n \,| \, \| x \|_\infty \leq 1,\,-2x_1+2 x_2 \leq 1 \}$ and $\U=[-1,1]$, and the sampling time $\Delta t = 1.0$. To check whether the continuous-time systems violates the constraint $-2x_1+2 x_2 \leq 1$ for 
$x_0 = (\, 0.6  \,\,\,\, 0.7  \,)^T$ and 
$u_0=1.0$, we solve~\eqref{eq:maxF}  with 
$h = (\,-2  \,\,\,\, 2  \,)^T$ and obtain $\ub{f}^\ast = \lb{f}^\ast= 1.5465$ using Alg.~\ref{alg:fMaxBranchAndBound}. Thus, the continuous-time systems violates the state constraints for some $t \in (0,\Delta t)$ as confirmed in Figs. \ref{fig:example1to4}.(e) and \ref{fig:example1to4}.(f). In contrast to Exmps.~\ref{exmp:1} and~\ref{exmp:2}, Alg.~\ref{alg:fMaxBranchAndBound} does not terminate without any bisection. In fact, as itemized in Tab.~\ref{tab:1to4}, we require eight bisections and the solution of three convex OP to identify $\ub{f}^\ast$ using the second overestimator proposed in Lem.~\ref{lem:Overestimators}.
A snapshot of the computed overestimators after three bisections is shown in Fig.~\ref{fig:example3Details}.
\end{exmp}

\begin{exmp}
\label{exmp:4}
We consider the system matrices
$$A = \begin{pmatrix} 0 & 6& 5\\ 5 & 1& 0 \\ 3& 2 &1   \end{pmatrix} \quad \text{and} \quad
B = \begin{pmatrix}  \textcolor{white}{+}1  \\ \textcolor{white}{+}0 \\-2  \end{pmatrix},$$
the constraints $\X = \{x \in \R^n \,| \, \| x \|_\infty \leq 0.2 \}$ and $\U=[-1,1]$, and the sampling time $\Delta t = 0.2$.
To check whether the continuous-time systems violates the constraint $x_3 \leq 0.2$ for 
$x_0 = (\, 2.6724  \,\,\,\, -2.3762 \,\,\,\,0.1105  \,)^T$ and 
$u_0=1.0$, we solve~\eqref{eq:maxF}  with 
$h = (\,0  \,\,\,\, 0 \,\,\,\,5  \,)^T$ and obtain $\ub{f}^\ast = 1.0000$ using Alg.~\ref{alg:fMaxBranchAndBound}.
In contrast to Exmps.~\ref{exmp:1} through~\ref{exmp:3}, the result $\ub{f}^\ast = 1.0000$ is not guaranteed to be exact. In fact, we obtain $\ub{f}^\ast-\lb{f}^\ast =  0.3123 \cdot 10^{-6}$ using the second overestimator in Lem.~\ref{lem:Overestimators}.   
The inexactness can be explained as follows. The example is constructed in such a way that $f^\ast=f(\Delta t)$ and $f^\prime(\Delta t)=f^{\prime\prime}(\Delta t)=0$. In other words, the maximum on $[0,\Delta t]$ is a saddle point of $f(t)$.
Thus, for any time-interval containing $\Delta t$, one of the interval inclusions $[f^{\prime}]$ and $[f^{\prime\prime}]$ has to be exact (at least $\lb{f}^\prime$ or $\ub{f}^{\prime\prime}$) in order to identify $f$ being non-decreasing or concave. However, since interval inclusions are inexact in general and in particular for this example, $\ub{f}^\ast$ has to be identified solely by using the overestimators $g$. 
Consequently, the number of required bisections is high compared to Exmps.~\ref{exmp:1} through~\ref{exmp:3} (see Tab.~\ref{tab:1to4}).
\end{exmp}

\begin{table}[htp]
\caption{Statistics on the application of Alg.~\ref{alg:fMaxBranchAndBound} to Exmps.~\ref{exmp:1} through~\ref{exmp:4}. For every example and every overestimator $g$ as in Lem.~\ref{lem:Overestimators}, we list the number of bisections and the number of solved convex OP necessary to identify $\ub{f}^\ast$. The itemized errors refer to $(\ub{f}^\ast - \lb{f}^\ast) \cdot 10^{6} $.}
\label{tab:1to4}
\begin{center}
\begin{tabular}{lccrccrcc}
\toprule
 & $g$ & \multicolumn{3}{c}{bisections} & \multicolumn{3}{c}{convex OP} & error \\
\midrule
 Exmp. 1 & 1--3 &  &0& & &1& & 0\\

\cmidrule{1-1}
 
Exmp. 2 & 1--3 &  &0& & &1& & 0\\

\cmidrule{1-1}
  & 1 &  &11& & &4& & 0\\
  Exmp. 3 & 2 &  &8& & &3& & 0\\
    & 3 &  &7& & &15& & 0\\

\cmidrule{1-1}

  & 1 &  &109& & &90& & 0.8149 \\
  Exmp. 4 & 2 &  &15& & &0& & 0.3123\\
    & 3 &  &14& & &29& & 0.3022 \\ 
    \bottomrule
\end{tabular}
\end{center}
\end{table}

\section{Conclusion}
\label{sec:Conclusion}

We presented a numerical method for the rigorous verification of constraint satisfaction for sampled linear systems.
In particular, we proposed a tailored branch and bound algorithm for the solution of the non-convex OP~\eqref{eq:maxF}
(resp.~\eqref{eq:conditionMaxSets}).
The core of the algorithm is a recently published procedure for the inclusion of interval matrix exponentials (see \cite{Goldsztejn2014}). 
Being able to solve~\eqref{eq:conditionMaxSets} for different $x_0$ and $u_0$ allows us to (offline) compute adapted state and input contstraints  according to \cite[Prop. 4 and Thm. 5]{Beradi2001_ECC} or \cite[Prop. 2]{SchulzeDarup2015_CDC1}. 
Satisfying these adapted constraints 
for the discretized system~\eqref{eq:sysDisc} finally guarantees  constraint satisfaction
of the continuous-time system~\eqref{eq:sysCont} w.r.t.~the original constraints~\eqref{eq:constraintsCont}.

The new method was illustrated with four examples. 
For every example, we were able to compute an $\epsilon$-optimal solution to the non-convex OP~\eqref{eq:maxF} (with $\epsilon=10^{-6}$). For three examples, the OP has even been solved exactly.
For the two technical examples taken from~\cite{Gutman1987} and~\cite{Sopasakis2014}, the algorithm terminated instantaneously without branching (i.e., without bisections). In fact, branching (and bounding) was only required for the two academic examples, which were designed to challenge Alg.~\ref{alg:fMaxBranchAndBound}. 
Such challenges are unlikely to appear in practice, however, since they were either caused by an inappropriately high sampling time $\Delta t$ (see Fig.~\ref{fig:example1to4}.(f)) or an extremely rare feature of $f$ in terms of a saddle point at the boundary of $[0,\Delta t]$ (see Fig.~\ref{fig:example1to4}.(g)).

Algorithm~\ref{alg:fMaxBranchAndBound} was particularly designed to solve problems of the form~\eqref{eq:maxF}. However, it can be used to solve any univariate OP on a convex domain, for which the objective function $f$ is of class $\C^2$ and for which interval inclusions for the first \textit{and} second derivative of $f$ can be computed efficiently. In this context, note that the list of suitable overestimators in Lem.~\ref{lem:Overestimators} is (by far) not complete. The overestimator of type 2, which performed most successfully for the analyzed examples (see Tab.~\ref{tab:1to4} and Fig.~\ref{fig:example3Details}) can for example be further improved using the results in~\cite{Sergeyev1998}.

\section{Acknowledgments}

Financial support by the German Research Foundation (DFG) through
the grant SCHU 2094/1-1 is gratefully acknowledged.

\section*{References}

\bibliographystyle{plain}

\end{document}